\documentclass[sn-mathphys-num]{sn-jnl}
\usepackage{graphicx}%
\usepackage{multirow}%
\usepackage{amsmath,amssymb,amsfonts}%
\usepackage{mathtools}
\usepackage[utf8]{inputenc}
\usepackage{amsthm}%
\usepackage{mathrsfs}%
\usepackage[title]{appendix}%
\usepackage{xcolor}%
\usepackage{textcomp}%
\usepackage{manyfoot}%
\usepackage{booktabs}%
\usepackage{algorithm}%
\usepackage{algorithmicx}%
\usepackage{algpseudocode}%
\usepackage{listings}%

\theoremstyle{thmstyleone}%
\newtheorem{theorem}{Theorem}
\newtheorem{lem}[theorem]{Lemma}
\newtheorem{corollary}{\sc Corollary}[section]
\theoremstyle{thmstyletwo}%
\newtheorem{example}{Example}%

\theoremstyle{thmstylethree}%
\newtheorem{definition}{Definition}%
\newcommand{\R}{\mathbb{R}}

\DeclareMathOperator{\D}{\mathcal{D}}
\DeclareMathOperator{\h}{\mathcal{H}}

\begin{document}

\title[Bi-slant Riemannian maps to Kenmotsu manifolds and some optimal inequalities]{Bi-slant Riemannian maps to Kenmotsu manifolds and some optimal inequalities}


\author[1]{\fnm{Adeeba} \sur{Zaidi}}\email{adeebazaidi.az25@gmail.com}
\author*[1]{\fnm{Gauree} \sur{Shanker}}\email{gauree.shanker@cup.edu.in}

\affil[1]{\orgdiv{Department of Mathematics and Statistics}, \orgname{Central University of Punjab}, \orgaddress{\city{Bathinda}, \postcode{151401}, \state{Punjab}, \country{India}}}


\abstract{	In this paper, we introduce bi-slant Riemannian maps from Riemannian manifolds to Kenmotsu manifolds which are the natural generalizations of invariant, anti-invariant, semi-invariant, slant, semi-slant and hemi-slant Riemannian maps, with nontrivial examples. We study these maps and give some curvature relations for $(rangeF_*)^\perp$. We construct Chen-Ricci inequalities, DDVV inequalities and further some optimal inequalities involving Casorati curvatures from bi-slant Riemannian manifolds to Kenmotsu space forms.}




\maketitle

\section{Introduction}

In the field of submanifold theory, it is quite interesting to know the relationship between intrinsic and extrinsic properties of a submanifold. In Riemannian geometry, the intrinsic properties of Riemannian manifolds are essentially the Riemannian invariants, which in turn influence the general form behavior of the manifold. In 1993, Chen \cite{C1} established the relationship between intrinsic and extrinsic invariants. This relationship takes the form of various inequalities and connects the main intrinsic and extrinsic invariants. Chen \cite{C2} also established a relationship between the squared mean curvature and Ricci curvature of a submanifold in the form of an inequality, known as the Chen-Ricci inequality. This was the beginning of the Chen invariants and Chen-Ricci relations, which have since become a topic of interest among geometers (see\cite{MM1, HK, MM2} etc). Later, In 2005, Chen proved optimal
relationship between Riemannian submersions and minimal immersions in \cite{C3}
and \cite{C4}. After that many optimal inequalities relations on Riemannian submersions are investigated by various geometers (see\cite{MA1, W1, O1}etc).\\
On the other hand, instead of concentrating on the sectional curvature with the extrinsic squared mean curvature, Casorati introduced Casorati curvature as the normalized square of the length of the second fundamental form. The notion of Casorati curvature extends the concept of the principal direction of a hypersurface of a Riemannian manifold. Given the critical role and significance of Casorati curvature in the visual perception of shapes and the appearance of manifolds, many authors have taken an interest in obtaining optimal inequalities for the Casorati curvatures of submanifolds in various ambient spaces (see \cite{A1,L2,LS} etc). 
In 2020, Lee et al \cite{LC} obtained optimal inequalities for Riemannian maps and Riemannian
submersions for complex space forms. Recently, G\"{u}nd\"{u}zalp et al studied Casorati inequalities for pointwise slant Riemannian maps \cite{G2}, pointwise hemi-slant Riemannian maps \cite{G1}, pointwise semi-slant riemannian maps \cite{G3} from complex space forms.\\
Motivated by these work, we study bi-slant Riemannian maps from Riemannian manifolds to Kenmotsu manifolds and established some inequalities.  
This article is organized as follows. Section 2 covers the basic terminologies, definitions, and properties of Riemannian maps and Kenmotsu manifolds that are essential throughout this article. Section 3 introduces the definition of bi-slant Riemannian maps from Riemannian manifolds to Kenmotsu manifolds, considering both cases when the Reeb vector field is in $rangeF_*$ and $(rangeF_*)^{\perp}$. Examples of such maps are constructed, followed by a study of some properties and an inequality related to these maps. In Section 4, we derive Ricci-Chen inequalities for bi-slant Riemannian maps from Riemannian manifolds to Kenmotsu space forms for both cases of Reeb vector fields and study DDVV inequality for such maps . In Section 5, we establish optimal inequalities for the same maps from Riemannian manifolds to Kenmotsu space forms involving Casorati curvatures, and we also discuss the inequalities for subcases for these maps.".

\section{Preliminaries}
In this section, basic definitions and terminologies are recalled that are needed throughout this article.\\
A $(2m+1)-$ dimensional differentiable manifold $M$ is said to have an almost-contact metric structure $(\psi,\xi,\eta, g)$ if it admits a tensor field $\psi$ of type $(1,1)$, a vector field (characteristic vector field or Reeb vector field) $\xi$, a  $1-$form $\eta$ and a Riemannian metric $g,$ satisfying \cite{B}
\begin{align}
	\psi^2 =-I + \eta \otimes \xi,&~~~  \psi \xi=0,~~\eta\circ\psi=0,~~ \eta(\xi)=1, \label{0}\\
	g(\psi X,\psi Y) \ &=\  g(X,Y)-\eta (X)\eta (Y),  \label{1} \\
	g(\psi X,Y) &= -g(X,\psi Y), \ \eta(X) \ =\ g(X,\xi) \label{2}
\end{align}
for any vector fields $X,Y \in \varGamma(TM).$ Moreover, $M$ with $(\psi,\xi,\eta, g)$ is called almost contact metric manifold. An almost contact metric manifold $M$ is called a Kenmotsu manifold, if \cite{K}
\begin{align}
	(\nabla_X\psi)Y \ & = \ g(\psi X,Y)\xi-\eta(Y)\psi X,\label{3} \\  \nabla_{X}\xi \ &=X-\eta(X)\xi, \label{4}
\end{align}
where $\nabla$ is the Levi-Civita connection of the Riemannian metric $g$ on $M$. Also, for a Kenmotsu manifold, we have 
\begin{align}\label{5}
	R(\psi X,\psi Y)Z-R(X,Y)Z=g(Y,Z)X-g(X,Z)Y+g(Y,\psi Z)\psi X-g(X,\psi Z)\psi Y.
\end{align}
A Kenmotsu manifold with constant $\psi$-holomorphic sectional curvature $c$ is called a Kenmotsu space form and denoted by $M(c)$. The curvature tensor $R$ of a Kenmotsu space form $M(c)$ is given by, \cite{K}
\begin{equation}\label{8}
	\begin{split}
		R(X,Y)Z&=\dfrac{c- 3}{4}
		\Big[g(Y,Z)X - g(X,Z)Y \Big]+\dfrac{c+1}{4}\Big[\eta(X)\eta(Z)Y - \eta(Y)\eta(Z)X\\& + \eta(Y )g(X,Z)\xi
		-\eta(X) g(Y,Z)\xi-g(\psi X,Z)\psi Y+g(\psi Y,Z)\psi X\\&+2g(\psi Y,X)\psi Z\Big]
	\end{split}
\end{equation} for all $X,Y\in \Gamma(TM).$\\
Further, consider $F: (\overline{M}^{\overline{m}},g_1)\rightarrow (M^m,g_2)$ be a smooth map between smooth finite dimensional Riemannian manifolds $\overline{M}^{\overline{m}}$ and $ M^m$ such that $0<rankF<min \{\overline{m},m\}$ and if $F_{*}: T_p\overline{M}\rightarrow T_{F(p)}M$ denotes the differential map at $p \in  \overline{M}$, and $ F(p)  \in  M$, then $T_p\overline{M}$ and $T_{F(p)}M$ split
orthogonally with respect to $ g_1(p)$ and $g_2(F(p))$, respectively, as \cite{F}
\begin{equation*} 
	T_p\overline{M} = kerF_{*p} \oplus (ker F_{*p})^\perp =\mathcal{V}_p\oplus \mathcal{H}_p,
\end{equation*}
where $\mathcal{V}_p= kerF_{*p}$ and $\mathcal{H}_p=(ker F_{*p})^\perp$ are vertical and horizontal parts of $T_{p}\overline{M}$ respectively. Since  $0<rankF<min \{\overline{m},m\}$, we have $(range F_{*p})^\perp \neq 0$; therefore, $T_{F(p)}M$ can be written as $	T_{F(p)}M = rangeF_{*p}  \oplus  (range F_{*p})^\perp,$
then the map $F : (\overline{M}^{\overline{m}}, g_1) \rightarrow (M^m, g_2)$ is called a Riemannian map at
$p  \in  \overline{M},$ if
$g_2(F_{*} X,F_{*} Y) \ =\ g_1(X,Y)$
for all vector fields $X, Y\in \Gamma (kerF_{*p})^\perp.$ The second fundamental form of the map $F$ is given by \cite{S7}
\begin{equation}  \label{10}
	(\nabla F_{*})(X,Y) = \overset{M}{\nabla^{F}_{X}}F_{*}Y-F_{*}(\nabla_X^{\overline{M}}Y), 
\end{equation}
where $\overset{M}{\nabla^{F}_{X}}F_{*}Y \circ F = \nabla ^{M}_{F_{*}X}F_{*}Y$, while $\nabla ^{M}$ is the Levi-Civita connection on $M$ and $\overset{M}{\nabla^{F} }$ is the pullback connection of $\nabla ^{M}$ along $F,$ provided that $(\nabla F_{*})(X,Y)$ has no components in $rangeF_*,$ if $X,Y\in \Gamma((kerF_*)^{\perp})$. \\
Let $\nabla^{F\perp}$ be a linear connection on $(F_*(T\overline{M}))^\perp$ and $\nabla_{X}^{F\perp}V$ be the orthogonal projection of $\nabla_X^MV$ onto $(rangeF_*)^\perp,~X\in\Gamma((kerF_*)^\perp),~ V\in \Gamma((rangeF_*)^\perp),$ then \cite{S7}
\begin{equation}\label{11}
	\nabla_X^MV=-S_VF_*X+\nabla_{X}^{F\perp}V,
\end{equation}
where $-S_VF_*X$ is the tangential component of $\nabla_X^MV.$ Noted that  $\nabla_X^MV$ is obtained from the pullback connection of $\nabla^M.$ It can be seen that $(S_VF_*X)_p$ depends only on $V_p$ and $F_{*p}(X_p),$ and $S_VF_*X$ is bi-linear in $V$ and $F_*X.$  By direct computations, we have
\begin{equation}\label{12}
	g_2(S_VF_*X,F_*Y)=g_2(V,(\nabla F_*)(X,Y))
\end{equation}
for $X,Y\in\Gamma((kerF_*)^\perp)$ and $V\in\Gamma((rangeF_*)^\perp).$ The curvature tensor $R^{F\perp}$ of the sub-bundle $(rangeF_*)^\perp$ is defined by $$R^{F\perp}(F_*(X),F_*(Y))V=\nabla_X^{F\perp}\nabla_{Y}^{F\perp}V-\nabla_Y^{F\perp}\nabla_{X}^{F\perp}V-\nabla_{[X,Y]}^{F\perp}V.$$
Further, let $R^{\overline{M}}, ~R^M$ and  $R^{F\perp}$ be the curvature tensors of $\nabla^{\overline{M}},~\nabla^M$ and $(rangeF_*)^\perp,$ respectively. Then we have the Gauss equation and the Ricci equation for the Riemannian map $F$ given by \cite{S7}
\begin{align}
	g_2\big(R^M(F_*X,F_*Y)F_*Z,F_*H\big)&=g_1(R^{\overline{M}}(X,Y)Z,H)+g_2\big((\nabla F_*)(X,Z),(\nabla F_*)(Y,H)\big)\notag \\&-g_2\big((\nabla F_*)(Y,Z),(\nabla F_*)(X,H)\big),\label{13}\\
	g_2\big(R^M(F_*X,F_*Y)V_1,V_2\big) &= g_2\big(R^{F\perp}(F_*X,F_*Y)V_1,V_2\big)+g_2\big([S_{V_2},S_{V_1}]F_*X,F_*Y\big),\label{14}
\end{align}
where $X,Y,Z, H\in\Gamma((kerF_*)^\perp)$, $V_1,V_2\in\Gamma((rangeF_*)^\perp)$ and $[S_{V_2},S_{V_1}]=S_{V_2}S_{V_1}-S_{V_1}S_{V_2}.$ Let $\{e_1,...,e_r\}$ and $\{v_{r+1},...,v_s\}$ be the orthonormal basis of $(kerF_*)^\perp_p$ and $(rangeF_*)^\perp_{F_*p},~p\in
\overline{M}$,  respectively. Then the Ricci curvature $\mathcal{R}ic^{(kerF_*)^{\perp}},$ scalar curvature $\tau^{(kerF_*)^{\perp}}$ and normalized scalar curvature $\varrho^{(kerF_*)^{\perp}}$ on $((kerF_*)^\perp)_p$ are given by 
\begin{align*}
	\mathcal{R}ic^{(kerF_*)^{\perp}}(X)&=\sum_{i=1}^{r}g_1(R^{\overline{M}}(e_i,X)X,e_i),~~\forall~X\in (kerF_*)^\perp,\\
	\tau^{(kerF_*)^{\perp}}&=\sum_{1\leq i<j\leq r}^{}g_1(R^{\overline{M}}(e_i,e_j)e_j,e_i),\\
	\varrho^{(kerF_*)^{\perp}} &= \dfrac{2\tau^{(kerF_*)^{\perp}}}{r(r-1)},
\end{align*}
whereas, a new scalar curvature $\tau^{\perp}$ of sub-bundle $(rangeF_*)^{\perp}$ called normal scalar curvature, is defines as
$$
	\tau^{(rangeF_*)^{\perp}}=\sqrt{\sum_{1\leq i<j\leq r}\sum_{1\leq \alpha<\beta\leq m}g_2(R^{\perp}(e_i,e_j)v_{\alpha},v_{\beta})^2}
$$
and normalized normal scalar curvature $\varrho^{(rangeF_*)^{\perp}}$ is given by
\begin{equation}\label{2.13}
	\varrho^{(rangeF_*)^{\perp}}=\dfrac{2\tau^{(rangeF_*)^{\perp}}}{r(r-1)}.
\end{equation}
Also, for the Riemannian map $F$, we can write \cite{C}
\begin{align}
	\zeta^{ \alpha}_{ij}&=g_2\big((\nabla F_*)(e_i,e_j),v_{\alpha}\big),~~i,j=1,...,r,~\alpha=r+1,...,m\big) \label{15}\\
	||\zeta||^2&=\sum_{i,j=1}^{r}g_2\big((\nabla F_*)(e_i,e_j),(\nabla F_*)(e_i,e_j)\big), \label{16}\\
	trace\zeta&=\sum_{i=1}^{r}(\nabla F_*)(e_i,e_i), \label{17}\\
	||trace\zeta||^2&=g_2(trace\zeta,trace\zeta) \label{18}.
\end{align}

 \section{Bi-slant Riemannian Maps to Kenmotsu manifolds}
In this section, we define bi-slant Riemannian maps from Riemannian manifolds to Kenmotsu manifolds. We study the geometry of these maps and construct some non trivial examples.
\begin{definition}
	Let  $F:(\overline{M},g_1)\rightarrow (M,\psi,\xi,\eta, g_2)$ be a Riemannian map from a Riemannian manifold $\overline{M}$ to Kenmotsu manifold $M$ of dimensions $\overline{m}~\text{and}~m$, respectively. The map $F$ is said to be bi-slant Riemannian map if there exists two slant distribution $\D_{\theta_1}$ and $\D_{\theta_2}$ in $rangeF_*$, such that
	\begin{enumerate}
		\item if $\xi\in \Gamma(rangeF_*),$
		$$rangeF_*=\D_{\theta_1}\oplus \D_{\theta_2}\oplus \{\xi\}.$$
		In this case, for $X\in \Gamma((kerF_*)^\perp),~~F_*(X)=\sigma F_*(X) + \pi F_*(X)+\eta(U)\xi;$
		\item if $\xi\in \Gamma((rangeF_*)^\perp),$
		$$rangeF_*=\D_{\theta_1}\oplus \D_{\theta_2}.$$
		In this case, for $X\in \Gamma((kerF_*)^\perp),~~F_*(X)=\sigma F_*(X) + \pi F_*(X),$
	\end{enumerate}
	where $\psi \D_{\theta_i} \perp \D_{\theta_j}$ for $i\neq j=1,2,$ and $\sigma, \pi$ are orthogonal projections of $F_*X$ on  $\D_{\theta_1}$ and $\D_{\theta_2},$ respectively. Whereas, $\theta_1,\theta_2$ are the slant angles for distributions $\D_{\theta_1}$ and $\D_{\theta_2},$ respectively.
\end{definition}
If the dimensions of  $\D_{\theta_1}$ and $\D_{\theta_2}$ are $2r_1$ and $2r_2$, then we have the following particular cases of Riemannian maps to Kenmotsu manifolds:
\begin{table}[htbp]
	\begin{tabular}{|c|c|c|c|c|}\hline
		\textbf{Cases}&\textbf{Dimensions of the distributions}&$\theta_{1}$&$\theta_{2}$&\textbf{Riemannian maps}\\ \hline
		1&$2r_1=0$ and $2r_2\neq 0$&-&0&Invariant\\ \hline
		2&$2r_1=0$ and $2r_2\neq 0$&-&$\pi/2$&Anti-invariant\\ \hline
		3& $2r_1,2r_2\neq 0$&0&$\pi/2$&Semi-invariant\\ \hline
		4&$2r_1=0$ and $2r_2\neq 0$&0&$\in (0,\pi/2)$&Proper slant\\ \hline
		5&$2r_1=02r_2\neq 0$&0&$\in (0,\pi/2)$&Semi-slant\\ \hline
		6& $2r_1=02r_2\neq 0$&$\pi/2$&$\in (0,\pi/2)$&Hemi-slant\\ \hline
	\end{tabular}
\end{table}\\
However, a bi-slant Riemannian map is said to be proper, if $0<\theta_1,\theta_2<\pi/2.$
Again, If $X\in \Gamma(kerF_*),$ then 
\begin{equation}\label{19}
	\psi F_*X=PF_*X+QF_*X,
\end{equation}
where, $PF_*X\in\Gamma(rangeF_*)$ and $QF_*X\in \Gamma((rangeF_*)^\perp)$.
If $V\in \Gamma((rangeF_*)^\perp),$ then 
\begin{equation}\label{3.2}
	\psi V=\phi V+\omega V,
\end{equation}
where $\phi U\in\Gamma(rangeF_*)$ and $\omega U\in \Gamma((rangeF_*)^\perp)$.
Also, $(rangeF_*)^\perp=Q\D_{\theta_1}\oplus~ Q\D_{\theta_2}\oplus~ \mu,$ where $\mu$ is is the complementary distribution to $Q\D_{\theta_1} \oplus~ Q\D_{\theta_2}$ in $(rangeF_*)^\perp.$
\begin{example}
	Let $M=\{(u_1,u_2,u_3,v_1,v_2,v_3,w)\in \R^7|~ w\neq 0\}$ be a Riemannian manifold with contact metric structure $(\psi,\xi,\eta, g_2)$ define by $\eta=dz~, \xi=\frac{\partial}{\partial w},\psi(\frac{\partial}{\partial u_1})=\frac{\partial}{\partial u_2},\psi(\frac{\partial}{\partial u_2})=-\frac{\partial}{\partial u_1},\psi(\frac{\partial}{\partial u_3})=\frac{\partial}{\partial v_3},\psi(\frac{\partial}{\partial v_3})=-\frac{\partial}{\partial u_3},\psi(\frac{\partial}{\partial v_1})=\frac{\partial}{\partial v_2}, \psi(\frac{\partial}{\partial v_2})=-\frac{\partial}{\partial v_1},\psi(\frac{\partial}{\partial w})=0$ and $g_2=du_1^2+\frac{4}{9}(du_2^2+du_3^2)+\frac{3}{16}(dv_2+3dv_2^2)+dv_3^2+dw^2$. Then, from \eqref{0}--\eqref{4} we can see that $M$ is a Kenmotsu manifold. Further, let $F:(\overline{M}=\R^7,g_1)\rightarrow (M,\psi,\xi,\eta, g_2)$ be a Riemannian map defines by \begin{align*}
		F_*(x_1,x_2,x_3,x_4,x_5,x_6,x_7)=&\big(x_1,x_3-\frac{1}{\sqrt{2}}x_4,\frac{1}{\sqrt{2}}x_3-\frac{1}{2}x_4,-\frac{1}{\sqrt{3}}x_5+x_6,\notag \\&\frac{1}{\sqrt{6}}x_5-\frac{1}{\sqrt{2}}x_6,0,x_7\big),
	\end{align*} where $\overline{M}$ is a Riemannian manifold with Cartesian co-ordinates $(x_1,x_2,x_3,x_4,x_5,x_6,x_7),$ and $g_1$ is a usual basis on $\R^7.$ Then, by simple calculation we get 
	$(kerF_*)^\perp=span\big\{X_1=\frac{\partial}{\partial x_1},X_2=\frac{1}{\sqrt{2}}\frac{\partial}{\partial x_3}-\frac{1}{2}\frac{\partial}{\partial x_4},X_3=\frac{1}{3}\frac{\partial}{\partial x_5}-\frac{1}{\sqrt{3}}\frac{\partial}{\partial x_6},X_4=\frac{\partial}{\partial x_7}\big\}$ and $rangeF_*=span\big\{F_*X_1=\frac{\partial}{\partial u_1},F_*X_2=\frac{3}{2\sqrt{2}}\frac{\partial}{\partial u_2}+\frac{3}{4}\frac{\partial}{\partial u_3},F_*X_3=-\frac{4}{3\sqrt{2}}\frac{\partial}{\partial v_1}+\frac{2}{3}\frac{\partial}{\partial v_2},F_*X_4=\frac{\partial}{\partial w}=\xi\big\},$ where $\D_{\theta_{1}}=span\{F_*X_1,F_*X_2\}$ and $\D_{\theta_{2}}=span\{F_*X_3\}.$
	Here, we can see $g_1(X_i,X_j)=g_2(F_*X_i,F_*X_j),i,j=1,...,4$ and $\theta_{1}=cos^{-1}\big(\frac{2}{3}\big) \text{and}~\theta_{2}=\frac{\pi}{2}.$ Hence, $F$ is a bi-slant Riemannian map with slant angles $cos^{-1}\big(\frac{2}{3}\big)$ and $\frac{\pi}{2}$. 
\end{example}
\begin{example}
	Let $M=\{(u_1,u_2,u_3,u_4,v_1,v_2,v_3,v_4,w)\in \R^9|~ w\neq 0\}$ be a Riemannian manifold with contact metric structure $(\psi,\xi,\eta, g_2)$ define by $\eta=dz,~ \xi=\frac{\partial}{\partial w},~\psi(\frac{\partial}{\partial u_1})=\frac{\partial}{\partial u_3},~\psi(\frac{\partial}{\partial u_3})=-\frac{\partial}{\partial u_1},~\psi(\frac{\partial}{\partial u_2})=\frac{\partial}{\partial u_4},~\psi(\frac{\partial}{\partial u_4})=-\frac{\partial}{\partial u_2},~\psi(\frac{\partial}{\partial v_1})=\frac{\partial}{\partial v_3}, ~\psi(\frac{\partial}{\partial v_3})=-\frac{\partial}{\partial v_1},~\psi(\frac{\partial}{\partial v_2})=\frac{\partial}{\partial v_4}, \psi(\frac{\partial}{\partial v_4})=-\frac{\partial}{\partial v_2},~\psi(\frac{\partial}{\partial w})=0$ and $g_2=\sum_{i=1}^{2}(du_i^2+dv_i^2)+\sum_{j=3}^{4}\frac{1}{2}\big(du_j^2+\frac{1}{(\beta^2+\gamma^2)}dv_j^2\big)+dw^2.$ Then, from \eqref{0}--\eqref{4} we can see that $M$ is a Kenmotsu manifold. Further, let $F:(\overline{M}=\R^7,g_1)\rightarrow (M,\psi,\xi,\eta, g_2)$ be a Riemannian map defines by
	 \begin{align*}
	F_*(x_1,x_2,x_3,x_4,x_5,x_6,x_7,x_8,x_9)=&\big(x_1,0,sin\alpha(x_3-x_4),cos\alpha(x_3-x_4),0,x_6,\notag \\&\beta(x_7+x_8),\gamma(x_7+x_8),x_9\big),
	\end{align*} where$~\alpha\in(0,\pi/2), 0\neq\beta,\gamma\in\R$ and  $\overline{M}$ is a Riemannian manifold with Cartesian co-ordinates $(x_1,x_2,x_3,x_4,x_5,x_6,x_7,x_8,x_9)$ and $g_2$ is usual inner product on $\R^9.$ Then by simple calculation, we get 
	$(kerF_*)^\perp=span\big\{X_1=\frac{\partial}{\partial x_1},X_2=\frac{\partial}{\partial x_3}-\frac{\partial}{\partial x_4},X_3=\frac{\partial}{\partial x_6},X_4=\frac{\partial}{\partial x_7}+\frac{\partial}{\partial x_8},X_5=\frac{\partial}{\partial x_9}\big\}$ and $rangeF_*=span\big\{F_*X_1=\frac{\partial}{\partial u_1},F_*X_2=2sin\alpha\frac{\partial}{\partial u_3}+2cos\alpha\frac{\partial}{\partial u_4},F_*X_3=\frac{\partial}{\partial v_2},F_*X_4=2\beta\frac{\partial}{\partial v_3}+2\gamma\frac{\partial}{\partial v_4},F_*X_5=\frac{\partial}{\partial w}=\xi\big\},$ where $\D_{\theta_{1}}=span\{F_*X_1,F_*X_2\}$ and $\D_{\theta_{2}}=span\{F_*X_3,F_*X_4\}.$
	Here, we can see $g_1(X_i,X_j)=g_2(F_*X_i,F_*X_j),i,j=1,...,5$ and $\theta_{1}=cos^{-1}sin\alpha~ \text{and}~\theta_{2}=cos^{-1}\Big(\frac{\gamma}{\sqrt{(\beta^2+\gamma^2)}}\Big).$ Hence, $F$ is a bi-slant Riemannian map with slant angles $cos^{-1}sin\alpha$ and $cos^{-1}\Big(\frac{\gamma}{\sqrt{(\beta^2+\gamma^2)}}\Big)$.
\end{example}
\begin{lem}
	Let $F:(\overline{M},g_1)\rightarrow (M,\psi,\xi,\eta, g_2)$ be a bi-slant Riemannian map from a Riemannian manifold $\overline{M}$ to Kenmotsu manifold $M$ with slant angles $\theta_{1}$ and $\theta_{2}$ for the distribution $\D_{\theta_{1}}$ and $\D_{\theta_{2}},$ respectively, then
	\begin{equation}\label{3.3}
		P^2F_*(X)=-cos^2\theta_i(\psi^2F_*X), ~i=1,2,
	\end{equation}
	where  $\theta_i=\begin{cases}
		\theta_{1},~ \text{if}~ F_*X\in\D_{\theta_{1}},\\
		\theta_{2}, ~\text{if}~ F_*X\in\D_{\theta_{2}}.
	\end{cases}$
\end{lem}
Using the above lemma we have the following result.
\begin{lem}
	Let $F:(\overline{M},g_1)\rightarrow (M,\psi,\xi,\eta, g_2)$ be a bi-slant Riemannian map from a Riemannian manifold $\overline{M}$ to Kenmotsu manifold $M$ with slant angles $\theta_{1}$ and $\theta_{2}$ for the distribution $\D_{\theta_{1}}$ and $\D_{\theta_{2}},$ respectively, then
	\begin{enumerate}
		\item $\phi QF_*X=sin^2\theta_i\psi^2F_*X,
		$
		\item $	QPF_*X+\omega QF_*X=0,
		$
		\item $g_2(PF_*X,PF_*Y)=cos^2\theta_ig_2(\psi F_*X,\psi F_*Y),
		$
		\item $g_2(QF_*X,QF_*Y)=sin^2\theta_ig_2(\psi F_*X,\psi F_*Y),
		$
		\item $g_2(\omega QF_*X,\omega QF_*Y)=sin^2\theta_icos^2\theta_ig_2(\psi F_*X,\psi F_*Y),
		$
	\end{enumerate}where $X,Y\in\Gamma((kerF_*)^{\perp})$ and  $\theta_i=\begin{cases}
		\theta_{1},~ \text{if}~ F_*X\in\D_{\theta_{1}},\\
		\theta_{2}, ~\text{if}~ F_*X\in\D_{\theta_{2}}
	\end{cases}.$
\end{lem}

\begin{theorem}
	Let $F:(\overline{M},g_1)\rightarrow (M,\psi,\xi,\eta, g_2)$ be a bi-slant Riemannian map from a Riemannian manifold $\overline{M}$ to Kenmotsu manifold $M$ with $\xi\in \Gamma(rangeF_*),$ then $rangeF_*$ defines a totally geodesic foliation if and only if
	\begin{align*}
		Q\big(S_{QF_*Y}F_*X-F_*(\nabla_X{}^*F_*PF_*Y))=\omega\big((\nabla F_*)(X,{}^*F_*PF_*Y)-\nabla_X^{F\perp}QF_*Y\big),
	\end{align*}
	where $X,Y\in\Gamma((kerF_*)^{\perp}).$
\end{theorem}
\begin{proof}
	Let $X,Y\in\Gamma((kerF_*)^{\perp}),$ since $F$ is a Riemannian map and $M$ is a Kenmotsu manifold, from \eqref{0}, \eqref{1}, \eqref{3}, \eqref{4} and \eqref{19}, we get
	\begin{align}\label{3.4}
		\nabla_{F_*X}F_*Y=-\psi\nabla_{F_*X}PF_*Y-\psi\nabla_{F_*X}QF_*Y-g_2(X,Y)\xi +F_*X\eta(F_*Y).
	\end{align}
	Let ${^*F_*}$ be the adjoint of the map $ F$, from \eqref{10}, \eqref{11} and \eqref{3.4}, we have
	\begin{align}\label{3.5}
		\nabla_{F_*X}F_*Y&=-\psi\big((\nabla F_*)(X,{}^*F_*PF_*Y)-F_*(\nabla_X{}^*F_*PF_*Y)\big)-\psi\big(-S_{QF_*Y}F_*X\notag \\&+\nabla_X^{F\perp}QF_*Y\big)-g_2(X,Y)\xi +F_*X\eta(F_*Y).
	\end{align}
	If  $\xi\in \Gamma(rangeF_*)$ and $(rangeF_*)$ defines a totally geodesic foliation, then from \eqref{3.5} we have the result.
\end{proof}
\begin{theorem}
	Let $F:(\overline{M},g_1)\rightarrow (M,\psi,\xi,\eta, g_2)$ be a bi-slant Riemannian map from a Riemannian manifold $\overline{M}$ to Kenmotsu manifold $M$ with $\xi\in \Gamma(rangeF_*),$ then $(rangeF_*)^{\perp}$ defines a totally geodesic foliation if and only if
	\begin{align*}
		g_2\big(S_{\omega V_1}F_*X,\phi V_2-F_*(\nabla{}^*F_*\phi V_1)\big)&=g_2\big([V_1,F_*X],V_2\big)+g_2\big((\nabla F_*)(X,{}^*F_*\phi V_1) \notag \\&+\nabla_{X}^{F\perp}\omega V_1,\omega V_2\big)
	\end{align*}
	where $X\in \Gamma((kerF_*)^{\perp})$ and $V_1,V_2\in\Gamma((rangeF_*)^{\perp}).$
\end{theorem}
\begin{proof}
	Let $V_1,V_2\in \Gamma(range)^\perp$ and $X\in  \Gamma((kerF_*)^\perp),$ then we have
	\begin{equation}\label{3.6}
		g_2(\nabla_{V_1}V_2,F_*X)=g_2\big([V_1,F_*X]+\nabla_{X}^{F}V_1,V_2\big).
	\end{equation}
	Since $M$ is a Sasakian manifold, from \eqref{1} and \eqref{3.6}, we have
	\begin{equation}\label{3.7}
		g_2(\nabla_{V_1}V_2,F_*X)=g_2\big([V_1,F_*X],V_2\big)+g_2(\nabla_{F_*X}\psi V_1,\psi V_2),
	\end{equation}
	if ${}^*F_*$ is a adjoint map of $F,$ using \eqref{10}, \eqref{11} and \eqref{3.2} in \eqref{3.7}, we get
	\begin{align}
		g_2(\nabla_{V_1}V_2,F_*X)&=g_2\big([V_1,F_*X],V_2\big)+g_2\big(F_*(\nabla{}^*F_*\phi V_1)-S_{\omega V_1}F_*X,\phi V_2\big)\notag \\&+g_2\big((\nabla F_*)(X,{}^*F_*\phi V_1)+\nabla_{X}^{F\perp}\omega V_1,\omega V_2\big).\label{3.8}
	\end{align}
	Hence, from \eqref{3.8}, we have the result.
\end{proof}

\begin{theorem}
	Let $F:(\overline{M},g_1)\rightarrow (M,\psi,\xi,\eta, g_2)$ be a bi-slant Riemannian map from a Riemannian manifold $\overline{M}$ to Kenmotsu manifold $M$ with $\xi\in \Gamma(rangeF_*),$ then $F$ is totally geodesic if and only if
	\begin{enumerate}
		\item [(1)]\begin{align*}
			(cos^2\theta_{2}-cos^2\theta_1)g_2(\nabla_X^F\sigma F_*Y,F_*Z)&=cos^2\theta_{2}g_2(\nabla_X^F F_*Y,V) +g_2(\nabla_X^{F\perp}QPF_*Y,V)\notag \\&+g_2(S_{ QF_*Y}F_*X,\phi V)+g_2(\nabla_X^{F\perp}QF_*Y,\omega V)
		\end{align*}
		\item [(2)] $kerF_*$ is totally geodesics.
		\item [(3)]$(kerF_*)^{\perp}$ is totally geodesics.
	\end{enumerate} 	where $X,Y\in \Gamma((kerF_*)^{\perp}).$
	
\end{theorem}
\begin{proof}
	Since $F$ is a Riemannian map and $M$ is a Kenmotsu manifold, from \eqref{3}, \eqref{4} and \eqref{10}, we can write
	\begin{align*}
		g_2((\nabla F_*)(X,Y),V)&=g_2(\nabla_X^FF_*Y,V)=-g_2(\psi\nabla_X^F\psi F_*Y,F_*Z),
	\end{align*}
	now, using \eqref{3}, \eqref{19} and \eqref{3.2} in above equation, we get
	\begin{align}
		g_2((\nabla F_*)(X,Y),V)&=-g_2(\nabla_{X}^FP^2F_*Y,V)+g_2(\nabla_X^{F\perp}QPF_*Y,V)+g_2(S_{ QF_*Y}F_*X,\phi V)\notag \\&+g_2(\nabla_X^{F\perp}QF_*Y,\omega V),\label{3.9}
	\end{align}
	substituting $F_*(Y)=\sigma F_*(Y)+\pi F_*(Y)+\eta(U)\xi$ in \eqref{3.9}, further using lemma 3.2 and simplifying, we get
	\begin{align*}
		g_2((\nabla F_*)(X,Y),V)&=cos^2\theta_1g_2(\nabla_{X}^F\sigma F_*Y,V)+cos^2\theta_2g_2(\nabla_{X}^F\pi F_*Y,V)\notag \\&+g_2(\nabla_X^{F\perp}QPF_*Y,V)+g_2(S_{ QF_*Y}F_*X,\phi V)+g_2(\nabla_X^{F\perp}QF_*Y,\omega V)\label{3.10}
	\end{align*}
	again, substituting $\pi F_*Y=F_*Y-\sigma F_*Y-\eta(F_*Y)\xi$ and using\eqref{4} in above equation, we have
	\begin{align}
		g_2((\nabla F_*)(X,Y),V)&=(cos^2\theta_{1}-cos^2\theta_2)g_2(\nabla_X^F\sigma F_*Y,V)\notag+cos^2\theta_{2}g_2(\nabla_X^F F_*Y,V)\notag \\&+g_2(\nabla_X^{F\perp}QPF_*Y,V)+g_2(S_{ QF_*Y}F_*X,\phi V)+g_2(\nabla_X^{F\perp}QF_*Y,\omega V),
	\end{align}
	that implies $(1)$, similarly we can obtain $(2)$ and $(3).$
\end{proof}

Further, Let $X,Y,Z,H\in \Gamma((kerF_*)^\perp)$ and $(rangeF_*)^\perp$ is totally geodesic, then we can write
\begin{align*}
	g_2(R^M(\psi F_*X,\psi F_*Y)QF_*Z,QF_*H)&=g_2(R^M(P F_*X,P F_*Y)QF_*Z,QF_*H) \notag \\&+g_2(R^M(Q F_*X,Q F_*Y)QF_*Z,QF_*H).
\end{align*}
Since $M$ is a Kenmotsu manifold, from \eqref{5} get
\begin{align*}
	g_2(R^M(\psi F_*X,\psi F_*Y)QF_*Z,QF_*H)&=g_2(R^M( F_*X, F_*Y)QF_*Z,QF_*H)\notag \\&-g_2(\psi F_*Y,QF_*Z)(\psi F_*X,QF_*H)\notag \\&+g_2(\psi F_*X,QF_*Z)(\psi F_*Y,QF_*H),
\end{align*}
using \eqref{14} in above equation, we get
\begin{align*}
	g_2(R^M(\psi F_*X,\psi F_*Y)QF_*Z,QF_*H)&=g_2\big(R^{F\perp}(F_*X,F_*Y)QF_*Z,QF_*H\big)\notag \\&+g_2\big([S_{QF_*H},S_{QF_*Z}]F_*X,F_*Y\big),\notag \\&-g_2(\psi F_*Y,QF_*Z)(\psi F_*X,QF_*H)\notag \\&+g_2(\psi F_*X,QF_*Z)(\psi F_*Y,QF_*H)
\end{align*}

and also from \eqref{14}, we have
\begin{align*}
	g_2(R^M(P F_*X,P F_*Y)QF_*Z,QF_*H)&= g_2\big(R^{F\perp}(PF_*X,PF_*Y)QF_*Z,QF_*H\big)\notag \\&+g_2\big([S_{QF_*H},S_{QF_*Z}]PF_*X,PF_*Y\big),
\end{align*}
hence, we conclude that
\begin{align}
	g_2(R^M(Q F_*X,Q F_*Y)QF_*Z,QF_*H)&=g_2\big(R^{F\perp}(F_*X,F_*Y)QF_*Z,QF_*H\big)\notag \\&-g_2\big(R^{F\perp}(PF_*X,PF_*Y)QF_*Z,QF_*H\big)\notag \\&+g_2\big([S_{QF_*H},S_{QF_*Z}]F_*X,F_*Y\big)\notag \\&-g_2\big([S_{QF_*H},S_{QF_*Z}]PF_*X,PF_*Y\big)\notag \\&-g_2(Q F_*Y,QF_*Z)g_2(Q F_*X,QF_*H)\notag \\&+g_2(Q F_*X,QF_*Z)g_2(Q F_*Y,QF_*H)\label{3.11}
\end{align}
\begin{theorem}
	Let $F:(\overline{M},g_1)\rightarrow (M,\psi,\xi,\eta, g_2)$ be a bi-slant Riemannian map from a Riemannian manifold $\overline{M}$ to Kenmotsu manifold $M$ with $\xi\in\Gamma(rangeF_*),~\mu=\{0\} ~\text{and}~(rangeF_*)^\perp$ is totally geodesic, then we have
	\begin{align}\label{3.12}
		\mathcal{R}ic^{(rangeF_*)^\perp}(V_1,V_2)&\geq g_2\big(R^{F\perp}(F_*e_j,F_*X)V_2,v_i\big)+g_2\big([S_{v_i},S_{V_2}]F_*e_j,F_*X\big)\notag \\& +(1-m+r)g_2(V_1,V_2),
	\end{align}where $X,Y\in\Gamma((kerF_*)^\perp)~\text{and}~V_1=QF_*X, V_2=QF_*Y\in\Gamma ((rangeF_*)^\perp).$ Also, $\{e_1,...,e_r\}$ and $\{v_{r+1},...,v_m\}$ be the orthonormal basis of $(kerF_*)^\perp_p$ and $(rangeF_*)^\perp_{F_*p},~p\in
	\overline{M}$, respectively such that $\{F_*e_1,...,F_*e_r\}$ is the orthonormal basis of $(rangeF_*)_{F_*p}.$ The equality case arises when $F$ is totally geodesic, in this case 
	\begin{align}\label{3.13}
		\mathcal{R}ic^{(rangeF_*)^\perp}(V_1,V_2)&=g_2\big(R^{F\perp}(F_*e_j,F_*X)V_2,v_i\big)-g_2\big(R^{F\perp}(PF_*e_j,PF_*X)V_2,v_i\big)\notag \\&+(1-m+r)g_2(V_1,V_2)
	\end{align}
\end{theorem}
\begin{proof}
	Since $\mu=\{0\}, $ then $(rangeF_*)^\perp=Q\D_{\theta_{1}}\oplus ~ Q\D_{\theta_{2}}.$ Also, $(rangeF_*)^\perp$ is totally geodesic, then from \eqref{3.11}, we have \eqref{3.12}. If $F$ is totally geodesic, then $S_VF_*X=0$, from \eqref{3.11} we have \eqref{3.13}.
\end{proof}
\begin{corollary}
	Let $F:(\overline{M},g_1)\rightarrow (M,\psi,\xi,\eta, g_2)$ be a bi-slant totally geodesic Riemannian map from a Riemannian manifold $\overline{M}$ to Kenmotsu manifold $M$ with $\xi\in\Gamma(rangeF_*),~\mu=\{0\}$, then we have \eqref{3.13}.
\end{corollary}
\begin{proof}
	Using theorem 3.7 and \eqref{3.11}, we have the required result.
\end{proof}

\section{Chen-Ricci and DDVV inequalities for bi-slant Riemannian maps to Kenmotsu space form}
In this section, we established Chen-Ricci inequalities for the bi-slant Riemannian maps from Riemannian manifolds to Kenmotsu space forms.\\ 
Let $F:\overline{M}^{\overline{m}}\rightarrow M(c)^m$ be a Riemannian map from a Riemannian manifold  to Kenmotsu space Form $M(c)$ with $\xi\in \Gamma(rangeF_*),$ where $\overline{m}$ and $m$ are the dimensions of $\overline{M}$ and $M$ respectively. Further, let $\{e_1,...,e_r\}$ and $\{v_{r+1},...,v_m\}$ be the orthonormal basis of $(kerF_*)^\perp_p$ and $(rangeF_*)^\perp_{F_*p},~p\in
\overline{M}$, respectively, then  $\{F_*e_1,...,F_*e_r\}$ is the orthonormal basis of $(rangeF_*)_{F_*p}$ such that $r=2r_1 +2r_2+1$ (if $\xi\in \Gamma(rangeF_*)$) or $r=2r_1 +2r_2$ (if $\xi\in \Gamma((rangeF_*)^\perp)$), where dimensions of  $\D_{\theta_1}$ and $\D_{\theta_2}$ are $2r_1$ and $2r_2,$ respectively. Then we can consider canonical bi-slant orthonormal frame as
$$F_*e_1,F_*e_2=sec\theta_1PF_*e_1,...,F_*e_{{2r_1}-1},F_*e_{2r_1}=sec\theta_1PF_*e_{{2r_1}-1},F_*e_{{2r_1}+1}, F_*e_{{2r_1}+2}$$
$$=sec\theta_2PF_*e_{{2r_1}+1},...,F_*e_{{2r_1}+{2r_2}-1}, F_*e_{{2r_1}+{2r_2}}=sec\theta_2PF_*e_{{2r_1}+{2r_2}-1},F_*e_{{2r_1}+{2r_2}+1}=\xi,$$
hence, we have
$$g_2^2(\psi F_*e_i,F_*e_{i+1})=\begin{cases}
	cos^2\theta_1,~~for ~~i=1,...,2r_1-1,\\
	cos^2\theta_2,~~for ~~i=2r_1+1,...,2r_1+2r_2-1,
\end{cases}$$
then,$$\sum_{i,j=1}^{r}g_2^2(\psi F_*e_i,F_*e_{i+1})=2(r_1cos^2\theta_1+r_2cos^2\theta_2).$$
Similarly, if $\xi\in \Gamma((rangeF_*)^\perp),$ we have same results for $g_2^2(\psi F_*e_i,F_*e_{i+1})$ and $\sum_{i,j=1}^{r}g_2^2(\psi F_*e_i,F_*e_{i+1}).$\\
\textbf{Remark:} Throughout this article, if $\xi\in \Gamma(rangeF_*),$ we consider $r=2r_1 +2r_2+1$ and if $\xi\in \Gamma((rangeF_*)^\perp)$ then $r=2r_1 +2r_2.$
\begin{theorem}
	Let $F:\overline{M}\rightarrow M(c)$ be a bi-slant Riemannian map from a Riemannian manifold $\overline{M}$ to Kenmotsu space Form $M(c)$ with $\xi\in \Gamma(rangeF_*)$ and $rankF=r<m$  then we have
	\begin{equation}\label{4.1}
		4\mathcal{R}ic^{(kerF_*)^{\perp}}(X)\leq (c-3)(r-1)-2(c+1)+||trace\zeta||^2+3(c+1)\sum_{i=1}^{r}g_2^2(\psi F_* X,F_*e_i),
	\end{equation}
	where $X$ is a unit horizontal vector field on $\overline{M},$ provided, $$g_2^2(\psi F_* e_1,e_i)=\begin{cases*}
		cos^2\theta_1, ~~if~~F_*e_1\in\D_{\theta_1}\\
		cos^2\theta_2,~~if~~F_*e_1\in\D_{\theta_2}
	\end{cases*}$$
	Moreover;\begin{enumerate}
		\item The equality case of \eqref{4.1} arises for  $X\in \Gamma((kerF_*)^\perp)$ if and only if
		\begin{equation}\label{4.2}
			\begin{cases}
				(\nabla F_*)(X,Y)=0~~\forall~Y\in\Gamma((kerF_*)^\perp)~ orthogonal~ to~ X,\\
				(\nabla F_*)(X,X)=\frac{1}{2}trace\zeta,
			\end{cases}
		\end{equation}
	\item The equality case of \eqref{4.1} true for all $X\in \Gamma((kerF_*)^\perp)$ if and only if $\zeta=0.$ or $r=3$ and $\zeta^{\alpha}_{11}=\zeta^{{\alpha}}_{22}, \alpha=r+1,...,m.$
	\end{enumerate}
\end{theorem}
\begin{proof}
	Let $M(c)$ be a Kenmotsu space form with constant sectional curvature $c,$ then from \eqref{8} and \eqref{13}, we have
	\begin{equation}
		\begin{split}
			g_1\big(R^{\overline{M}}(X,Y)Z,H\big)&=\dfrac{c- 3}{4}
			\Big\{g_1(Y,Z)g_1(X,H) - g_1(X,Z)g_1(Y,H) \Big\}\\&+\dfrac{c+1}{4}\Big\{\eta(F_*X)\eta(F_*Z)g_1(Y,H) - \eta(F_*Y)\eta(F_*Z)g_1(X,H)\\& + \eta(F_*Y )\eta(F_*H)g_1(X,Z)
			-\eta(F_*X)\eta(F_*H) g_1(Y,Z)\\&-g_2(\psi F_*X,F_*Z)g_2(\psi F_* Y,F_*H)+g_2(\psi F_*Y,F_*Z)g_2(\psi F_*X,F_*H)\\&+2g_2(\psi F_*Y,F_*X)g_2(\psi F_*Z,F_*H)\Big\}-g_2\big((\nabla F_*)(X,Z),(\nabla F_*)(Y,H)\big)\\& +g_2\big((\nabla F_*)(Y,Z),(\nabla F_*)(X,H)\big), \label{4.3}
		\end{split}
	\end{equation}
	where $X,Y,Z,H\in \Gamma((kerF_*)^\perp).$ 
	let $\{e_1,...,e_{r_1},e_{r_1+1},...,e_{2r_1+2r_2+1=r}$ with $e_1=X$ and $\{v_{r+1},...,v_s\}$ be the orthonormal basis of $(kerF_*)^\perp_p$ and $(rangeF_*)^\perp_{F_*p},~p\in\overline{M}$, respectively, where $dim(\D_{\theta_{1}})=2r_1$ and $dim(\D_{\theta_{2}})=2r_2,$ then  $\{F_*e_1,...,F_*e_{2r_1},F_*e_{2r_1+1},...,F_*e_{2r_1+2r_2+1=r}=\xi\}$ is the orthonormal basis of $(rangeF_*)_{F_*p}$. From \eqref{15}--\eqref{18} and \eqref{4.3}, we get
		\begin{equation}\label{4.4}
		\dfrac{c-3}{4}r(r-1)+\dfrac{c+1}{4}\Big\{-2r+2+3\sum_{i,j=1}^{r}g_2^2(\psi F_*e_i,F_*ej)\Big\}=||\zeta||^2-||trace\zeta||^2-2\tau^{(kerF_*)^{\perp}}(p)
	\end{equation}
	Also, we can see that \cite{C}
	\begin{equation}\label{4.5}
		\begin{split}
			||\zeta||^2&=\dfrac{1}{2}||trace\zeta||^2+\dfrac{1}{2}\sum_{\alpha =r+1}^{m}\big(\zeta^{{\alpha}}_{11}-\zeta^{{\alpha}}_{22}-...-\zeta^{{\alpha}}_{rr}\big)^2+2\sum_{\alpha=r+1}^{m}\sum_{i=2}^{r}\big(\zeta^{{\alpha}}_{1i}\big)^2\\&-2\sum_{\alpha =r+1}^{m}\sum_{2\leq i<j\leq r}\Big\{\zeta^{{\alpha}}_{ii}\zeta^{{\alpha}}_{jj}-\big(\zeta^{{\alpha}}_{ij}\big)^2\Big\},~~~6(r_1cos^2\theta_1+r_2cos^2\theta_2)
		\end{split}
	\end{equation}
	and 
	\begin{equation}\label{4.6}
		\begin{split}
			\tau^{(kerF_*)^{\perp}}(p)&=\dfrac{c-3}{8}(r-1)(r-2)+\dfrac{c+1}{4}\Big\{-r+2+3\sum_{2\leq i<j\leq r}g_2^2(\psi F_*e_i,F_*ej)\Big\}\\&-\sum_{\alpha =r+1}^{m}\sum_{2\leq i<j\leq r}\Big\{\big(\zeta^{{\alpha}}_{ij}\big)^2-\zeta^{{\alpha}}_{ii}\zeta^{{\alpha}}_{jj}\Big\}+\mathcal{R}ic^{(kerF_*)^{\perp}}(X),
		\end{split}
	\end{equation}
	substituting the value of \eqref{4.5} and \eqref{4.6} in \eqref{4.4}, we get
	\begin{equation}\label{4.7}
		\begin{split}
			\mathcal{R}ic^{(kerF_*)^{\perp}}(X)&=\dfrac{c-3}{4}(r-1)-\dfrac{c+1}{2}+\dfrac{1}{4}||trace\zeta||^2+\dfrac{3(c+1)}{4}\sum_{i=1}g_2^2(\psi F_* X,F_*e_i)\\&-\dfrac{1}{4}\sum_{\alpha =r+1}^{m}\Big(\zeta^{{\alpha}}_{11}-\zeta^{{\alpha}}_{22}-...-\zeta^{{\alpha}}_{rr}\Big)^2-\sum_{\alpha=r+1}^{m}\sum_{i=2}^{r}\big(\zeta^{{\alpha}}_{1i}\big)^2\\
			&\leq \dfrac{c-3}{4}(r-1)-\dfrac{c+1}{2}+\dfrac{1}{4}||trace\zeta||^2+\dfrac{3(c+1)}{4}\sum_{i=1}g_2^2(\psi F_* X,F_*e_i),
		\end{split}
	\end{equation}
	hence, from we have \eqref{4.1}.
	The equality case of \eqref{4.1} holds if and only if\\
	$$\zeta^{{\alpha}}_{11}=\zeta^{{\alpha}}_{22}+...+\zeta^{{\alpha}}_{r-1r-1}$$
	$$and\quad \zeta^{{\alpha}}_{1i}=0,~~~~i=1,...,r-1; \alpha=r+1,...,m;$$
	hence we have \eqref{4.2}. Further, we consider that the equality case of \eqref{4.1} holds for all $X\in \Gamma((kerF_*)^\perp),$ then for all $\alpha,$ we have $$\zeta^{{\alpha}}_{ij}=0,\quad i\neq j,$$
	\begin{align*}
		2\zeta^{{\alpha}}_{ii}&=\zeta^{{\alpha}}_{11}+\zeta^{{\alpha}}_{22}+...+\zeta^{{\alpha}}_{r-1r-1}, ~~i=1,...,r-1;\\
		or\quad	(r-3&)\Big(\zeta^{{\alpha}}_{11}+\zeta^{{\alpha}}_{22}+...+\zeta^{{\alpha}}_{r-1r-1}\Big)=0,
	\end{align*}
	which implies that either $r=3$ or $(\zeta^{{\alpha}}_{11}+\zeta^{{\alpha}}_{22}+...+\zeta^{{\alpha}}_{r-1r-1}=0.$ If $r=3,$ then $\zeta^{{\alpha}}_{11}=\zeta^{{\alpha}}_{22}+\zeta^{{\alpha}}_{33}.$ Also, assuming $F_*e_3=\xi,$ we get $\zeta^{{\alpha}}_{33}=0.$
	Hence, the equality case of \eqref{4.1} holds for all $X\in \Gamma((kerF_*)^\perp),$ either $r=3$ and $\zeta^{{\alpha}}_{11}=\zeta^{{\alpha}}_{22}$ or $\zeta^{}=0.$
	Since the converse case is trivial so we have the result.
\end{proof}
Similarly, when $\xi\in \Gamma((rangeF_*)^\perp)$ we have the following result.
\begin{theorem}
	Let $F:\overline{M}\rightarrow M(c)$ be a bi-slant Riemannian map from a Riemannian manifold $\overline{M}$ to Kenmotsu space Form $M(c)$ with $\xi\in \Gamma((rangeF_*)^\perp)$ and $rankF=r<m$  then we have
	\begin{equation}\label{4.8}
		4	\mathcal{R}ic^{(kerF_*)^{\perp}}(X)\leq (c-3)(r-1)+||trace\zeta||^2+3(c+1)\sum_{i=1}^{r}g_2^2(\psi F_* X,F_*e_i),
	\end{equation}
	where $X$ is a unit horizontal vector field on $\overline{M},$ provided, $$g_2^2(\psi F_* e_1,e_i)=\begin{cases*}
		cos^2\theta_1, ~~if~~F_*e_1\in\D_{\theta_1},\\
		cos^2\theta_2,~~if~~F_*e_1\in\D_{\theta_2}.
	\end{cases*}$$
	Moreover,\begin{enumerate}
		\item The equality case of \eqref{4.8} arises for  $X\in \Gamma((kerF_*)^\perp)$ if and only if
		\begin{equation}\label{4.9}
			\begin{cases}
				(\nabla F_*)(X,Y)=0~~\forall~Y\in\Gamma((kerF_*)^\perp)~ orthogonal~ to~ X,\\
				(\nabla F_*)(X,X)=\frac{1}{2}trace\zeta.
			\end{cases}
		\end{equation}
		\item The equality case of \eqref{4.8} true for all $X\in \Gamma((kerF_*)^\perp)$ if and only if $\zeta=0$ or $dim((kerF_*)^{\perp})=2$ and $\zeta^{{\alpha}}_{11}=\zeta^{{\alpha}}_{22}, \alpha=r+1,...,m.$
	\end{enumerate}
	\subsection*{DDVV inequalities from Riemannian manifolds to Kenmotsu space forms}
	\begin{theorem}
		Let $F:(\overline{M},g_1)\rightarrow M(c)$ be a bi-slant Riemannian map from a Riemannian manifold $\overline{M}$ to Kenmotsu space form M(c) with $\xi\in \Gamma(rangeF_*)$ and $rankF=r<m$ , then normalized normal scalar curvature $\varrho^{(rangeF_*)^{\perp}}$ and normalized scalar curvature $\varrho^{(kerF_*)^{\perp}}$ satisfy
		\begin{equation}
			\varrho^{(rangeF_*)^{\perp}}+\varrho^{(kerF_*)^{\perp}}\leq \dfrac{1}{r^2}||trace\zeta||^2+\dfrac{c-3}{4}-\dfrac{c+1}{2r}+\dfrac{3(c+1)}{2r(r-1)}(r_1cos^2\theta_{1}+r_2cos^2\theta_{2}),
		\end{equation}
		where  $\theta_{1},\theta_{2}$ are the slant angles and $2r_1, 2r_2$ are the dimensions of distributions $\D_{\theta_{1}} \text{and } \D_{\theta_{2}},$ respectively.
	\end{theorem}
	\begin{proof}
		Let $\{e_1,...,e_{2r_1},e_{2r_1+1},...,e_{2r_1+2r_2+1=r}$ and $\{v_{r+1},...,v_s\}$ be the orthonormal basis of $(kerF_*)^\perp_p$ and $(rangeF_*)^\perp_{F_*p},~p\in\overline{M}$, respectively, where $dim(\D_{\theta_{1}})=2r_1$ and $dim(\D_{\theta_{2}})=2r_2,$ then  $\{F_*e_1,...,F_*e_{2r_1},F_*e_{2r_1+1},...,F_*e_{2r_1+2r_2+1=r}=\xi\}$ is the orthonormal basis of $(rangeF_*)_{F_*p}$. Then from \eqref{8}, \eqref{13} and \eqref{14}, we have
		\begin{align*}
			\sum_{ \alpha,\beta=r+1}^{m}\sum_{i,j=1}^{r}	g_2\big(R^{F\perp}(F_*e_i,F_*e_j)V_{\alpha},V_{\beta}\big)=-\sum_{ \alpha,\beta=r+1}^{m}\sum_{i,j=1}^{r}g_2\big([S_{v_{\beta}},S_{v_{\alpha}}]F_*e_i,F_*e_j\big)
		\end{align*}
		using the above result in \eqref{2.13}, we get
		\begin{equation}\label{4.11}
			\varrho^{(rangeF_*)^{\perp}}=\dfrac{2}{r(r-1)}\sqrt{\sum_{1\leq i<j\leq r}\sum_{1\leq \alpha<\beta\leq m}g_2^2\big([S_{v_{\beta}},S_{v_{\alpha}}]F_*e_i,F_*e_j\big)},
		\end{equation}
		using \eqref{12} in \eqref{4.11} and solving further, we get
		\begin{equation}\label{4.12}
			\varrho^{(rangeF_*)^{\perp}}=\dfrac{2}{r(r-1)}\sqrt{\sum_{1\leq i<j\leq r}\sum_{1\leq \alpha<\beta\leq m}\Big(\sum_{k=1}^{r}(\zeta_{jk}^{\alpha}\zeta_{ik}^{\beta}-\zeta_{ik}^{\alpha}\zeta_{jk}^{\beta})\Big)^2},
		\end{equation}
		Further, from \cite{LZ} (page no. 2), we have
		\begin{align*}
			\sum_{\alpha =r+1}^{m}\sum_{1\leq i<j\leq r}(h_{ii}^{\alpha}-h_{jj}^{\alpha})^2&+2r\sum_{\alpha =r+1}^{m}\sum_{1\leq i<j\leq r}(h_{ij}^{\alpha})^2\notag\\&\geq 2r\Bigg(\sum_{1\leq \alpha<\beta\leq m}\sum_{1\leq i<j\leq r}\Big(\sum_{k=1}^{r}(h_{jk}^{\alpha}h_{ik}^{\beta}-h_{ik}^{\alpha}h_{jk}^{\beta})\Big)^2\Bigg)^{\frac{1}{2}},
		\end{align*}
		where $(h^{\alpha}_{ij})(i,j = 1,...,n \text{ and } \alpha = r+1,...,m)$ be the entries of the second fundamental form of manifold $B,$ an n-dimensional manifold isometrically immersed into the space form $\overline{B}^{n+m}(c)$ of constant sectional curvature $c,$ under the orthonormal frames of both the tangent bundle and the normal bundle. It can be verify that above inequality is also true for $\zeta^{ \alpha}_{ij}.$ Hence using the above inequality for $\zeta^{ \alpha}_{ij}$ in \eqref{4.12}, we get
		\begin{equation}\label{4.13}
			r^2(r-1)\varrho^{(rangeF_*)^{\perp}}\leq \sum_{\alpha =r+1}^{m}\sum_{1\leq i<j\leq r}(\zeta_{ii}^{\alpha}-\zeta_{jj}^{\alpha})^2+2r\sum_{\alpha =r+1}^{m}\sum_{1\leq i<j\leq r}(\zeta_{ij}^{\alpha})^2.
		\end{equation}
		Further, we can compute
		\begin{equation}\label{4.14}
			(r-1)||trace\zeta||^2=\sum_{\alpha=r+1}^{m}\sum_{i,j=1}^{r}\zeta_{ii}^{\alpha}\zeta_{jj}^{\alpha}=\sum_{\alpha =r+1}^{m}\sum_{1\leq i<j\leq r}\Big(\big(\zeta_{ii}^{\alpha}-\zeta_{jj}^{\alpha}\big)^2+2r\zeta_{ii}^{\alpha}\zeta_{jj}^{\alpha}\Big),
		\end{equation}
		and 
		\begin{align}\label{4.15}
			\tau^{(kerF_*)^{\perp}}&=\dfrac{(r-1)}{4}\Big(\dfrac{c-3}{2}r-c-1\Big)+\dfrac{3(c+1)}{4}(r_1cos^2\theta_{1}+r_2cos^2\theta_{2})\notag\\&+\sum_{\alpha =r+1}^{m}\sum_{1\leq i<j\leq r}\Big(\zeta_{ii}^{\alpha}\zeta_{jj}^{\alpha}-\big(\zeta_{ij}^{\alpha}\big)^2\Big).
		\end{align}
		From \eqref{4.13}, \eqref{4.14} and \eqref{4.15}, we get the required result.
	\end{proof}
\end{theorem}
Similarly, when $\xi\in \Gamma((rangeF_*)^\perp)$ we have the following result.
\begin{theorem}
	Let $F:(\overline{M},g_1)\rightarrow M(c)$ be a bi-slant Riemannian map from a Riemannian manifold $\overline{M}$ to Kenmotsu space form M(c) with $\xi\in \Gamma((rangeF_*)^{\perp})$ and $rankF=r<m$ , then normalized normal scalar curvature $\varrho^{(rangeF_*)^{\perp}}$ and normalized scalar curvature $\varrho^{(kerF_*)^{\perp}}$ satisfy
	\begin{equation*}
		\varrho^{(rangeF_*)^{\perp}}+\varrho^{(kerF_*)^{\perp}}\leq \dfrac{1}{r^2}||trace\zeta||^2+\dfrac{c-3}{4}+\dfrac{3(c+1)}{2r(r-1)}(r_1cos^2\theta_{1}+r_2cos^2\theta_{2}),
	\end{equation*}
	where  $\theta_{1},\theta_{2}$ are the slant angles and $2r_1, 2r_2$ are the dimensions of distributions $\D_{\theta_{1}} \text{and } \D_{\theta_{2}},$ respectively.
\end{theorem}
\begin{corollary}
	Let $F:(\overline{M},g_1)\rightarrow M(c)$ be a bi-slant Riemannian map from a Riemannian manifold $\overline{M}$ to Kenmotsu space form M(c) with $\xi\in \Gamma(rangeF_*)$ and $rankF=r<m,$ then for normalized normal scalar curvature $\varrho^{(rangeF_*)^{\perp}}$ and normalized scalar curvature $\varrho^{(kerF_*)^{\perp}},$ we have the following relations
	\begin{table}[htbp]
		\begin{tabular}{|c|c|c|c|p{6cm}|}\hline
			\textbf{Riemannian Maps}&\textbf{$r_1$\& $r_2$}&$\theta_{1}$ &$\theta_{2}$&$	\varrho^{(rangeF_*)^{\perp}}+\varrho^{(kerF_*)^{\perp}}$\\ \hline
			Invariant&$r_1=0$&-&$0$&$ \leq\dfrac{1}{r^2}||trace\zeta||^2+\frac{c-3}{4}+\frac{c+1}{r},$\\
			\hline
			Anti-invariant&$r_1=0$&-&$\frac{\pi}{2}$&$\leq\dfrac{1}{r^2}||trace\zeta||^2+\frac{c-3}{4},$\\ \hline
			Semi-invariant&$r_1,r_2\neq0$&$0$&$\frac{\pi}{2}$&$\leq\dfrac{1}{r^2}||trace\zeta||^2+\frac{c-3}{4}+\frac{c+1}{2r}\big\{\frac{3r_1}{r-1}-\frac{1}{2}\big\},$\\ \hline
			Proper slant&$r_1=0$&-&$\in(0,\frac{\pi}{2})$&$\leq\dfrac{1}{r^2}||trace\zeta||^2+\frac{c-3}{4}+\frac{c+1}{4r}(3cos^2\theta_{2}-1),$\\ \hline
			Semi-slant&$r_1,r_2\neq0$&$0$&$\in(0,\frac{\pi}{2})$&$\leq\dfrac{1}{r^2}||trace\zeta^{\h}||^2+\frac{c-3}{4}+\frac{c+1}{2r}\big\{\frac{3(r_1+r_2cos^2\theta_{2})}{r-1}-\frac{1}{2}\big\},$\\ \hline
			Hemi-slant&$r_1,r_2\neq0$&$\frac{\pi}{2}$&$\in(0,\frac{\pi}{2})$&$\leq\dfrac{1}{r^2}||trace||^2+\frac{c-3}{4}+\frac{c+1}{2r}\big\{\frac{3r_2cos^2\theta_{2}}{r-1}-\frac{1}{2}\big\},$
			\\ \hline
		\end{tabular}
	\end{table}\\
		where  $\theta_{1},\theta_{2}$ are the slant angles and $2r_1, 2r_2$ are the dimensions of distributions $\D_{\theta_{1}} \text{and } \D_{\theta_{2}},$ respectively.
\end{corollary}

\begin{corollary}
	Let $F:\overline{M}\rightarrow M(c)$ be a Riemannian map from a Riemannian manifold $\overline{M}$ to Kenmotsu space form $M(c)$ with $\xi\in \Gamma((rangeF_*)^{\perp})$ and $rankF=r<m$, then for normalized normal scalar curvature $\varrho^{(rangeF_*)^{\perp}}$ and normalized scalar curvature $\varrho^{(kerF_*)^{\perp}},$ we have the following relations
	\begin{table}[htbp]
		\begin{tabular}{|c|c|c|c|p{6.1cm}|}\hline
			\textbf{Riemannian Maps}&\textbf{$r_1$\& $r_2$}&$\theta_{1}$ &$\theta_{2}$&$	\varrho^{(rangeF_*)^{\perp}}+\varrho^{(kerF_*)^{\perp}}$\\ \hline
			Invariant&$r_1=0$&-&$0$&$ \leq\dfrac{1}{r^2}||trace\zeta||^2\frac{c-3}{4}+\frac{3(c+1)}{4(r-1)},$\\
			\hline
			Anti-invariant&$r_1=0$&-&$\frac{\pi}{2}$&$\leq\dfrac{1}{r^2}||trace\zeta||^2+\frac{c-3}{4},$\\ \hline
			Semi-invariant&$r_1,r_2\neq0$&$0$&$\frac{\pi}{2}$&$\leq\dfrac{1}{r^2}||trace\zeta||^2+\frac{c-3}{4}+\frac{3(c+1)r_1}{2r(r-1)},$\\ \hline
			Proper slant&$r_1=0$&-&$\in(0,\frac{\pi}{2})$&$\leq\dfrac{1}{r^2}||trace\zeta||^2+\frac{c-3}{4}+\frac{3(c+1)}{4(r-1)}cos^2\theta_{2},$\\ \hline
			Semi-slant&$r_1,r_2\neq0$&$0$&$\in(0,\frac{\pi}{2})$&$\leq\dfrac{1}{r^2}||trace\zeta||^2+\frac{c-3}{4}+\frac{3(c+1)}{2r(r-1)}(r_1+r_2cos^2\theta_{2}),$\\ \hline
			Hemi-slant&$r_1,r_2\neq0$&$\frac{\pi}{2}$&$\in(0,\frac{\pi}{2})$&$\leq\dfrac{1}{r^2}||trace\zeta||^2+\frac{c-3}{4}+\frac{3(c+1)}{2r(r-1)}r_2cos^2\theta_{2},$
			\\ \hline
		\end{tabular}
	\end{table}\\
		where  $\theta_{1},\theta_{2}$ are the slant angles and $2r_1, 2r_2$ are the dimensions of distributions $\D_{\theta_{1}} \text{and } \D_{\theta_{2}},$ respectively.
\end{corollary}

\section{Optimal inequalities for bi-slant Riemannian maps to Kenmotsu space form involving Cosorati curvature}
In this section, we obtain inequalities for bi-slant Riemannian maps from Riemannian manifolds to Kenmotsu space forms involving Cosorati curvatures in both cases when $\xi\in \Gamma(rangeF_*)$ or $\xi\in \Gamma((rangeF_*)^\perp).$\\ \\
Let $\mathcal{C}^{(kerF_*)^{\perp}}$ be the second fundamental form of the horizontal space $(kerF_*)^{\perp}$ over manifold $(M,\psi,\eta,\xi,g_2)$, called the Casorati curvature
of $(kerF_*)^{\perp}.$ Hence, we have
\begin{equation}\label{5.10}
	\mathcal{C}^{(kerF_*)^{\perp}}=\dfrac{1}{r}\sum_{\alpha=r+1}^{m}\sum_{i,j=1}^{r}\Big(\zeta^{{\alpha}}_{ij}\Big)^2.	
\end{equation}
Now, consider that $\mathcal{L}^{(kerF_*)^{\perp}}$ be the $s$ dimensional subspace of  horizontal space $(kerF_*)^{\perp}, s\geq 2,$ with an orthogonal basis $\{e_1,...,e_s\},$ then the Cosorati curvature $\mathcal{C}^{(kerF_*)^{\perp}}\big(\mathcal{L}^{(kerF_*)^{\perp}}\big)$ of $\mathcal{L}^{(kerF_*)^{\perp}}$ is defined as
\begin{equation}\label{5.11}
	\mathcal{C}^{(kerF_*)^{\perp}}\big(\mathcal{L}^{(kerF_*)^{\perp}}\big)=\dfrac{1}{s}\sum_{\alpha =r+1}^{m}\sum_{i,j=1}^{s}\Big(\zeta^{{\alpha}}_{ij}\Big)^2,
\end{equation}
and the normalized $\delta^{(kerF_*)^{\perp}}$-Cosorati curvatures $\delta^{(kerF_*)^{\perp}}_{\mathcal{C}}(r-1)$ and $\hat{\delta}^{(kerF_*)^{\perp}}_{\mathcal{C}}(r-1)$ of  $(kerF_*)^{\perp}$ are defined as
\begin{align}\label{5.12}
	[\delta^{(kerF_*)^{\perp}}_{\mathcal{C}}(r-1)]_p&=\dfrac{1}{2}\mathcal{C}^{(kerF_*)^{\perp}}_p+\dfrac{r+1}{2r}inf\big\{\mathcal{C}^{(kerF_*)^{\perp}}\big(\mathcal{L}^{(kerF_*)^{\perp}}\big)|\mathcal{L}^{(kerF_*)^{\perp}} \notag\\&\text{a hyperplane of $(kerF_*)^\perp$}\big\}
\end{align}
\begin{align}\label{5.13}
	[\hat{\delta}^{(kerF_*)^{\perp}}_{\mathcal{C}}(r-1)]_p&=2\mathcal{C}^{(kerF_*)^{\perp}}_p-\dfrac{2r-1}{2r}inf\big\{\mathcal{C}^{(kerF_*)^{\perp}}\big(\mathcal{L}^{(kerF_*)^{\perp}}\big)|\mathcal{L}^{(kerF_*)^{\perp}} \notag\\&\text{a hyperplane of $(kerF_*)^\perp$}\big\}
\end{align}
\begin{lem}\cite{TM}
	Let $\Xi=\{(x_1,...,x_n)\in\R^n:x_1+...+x_n=k\}$ be a hyperplane of $\R^n$ and $f: \R^n\rightarrow \R$ be a quadratic form given by
	\begin{equation*}
		f(x_1,...,x_n)=b\sum_{i=1}^{n-1}(x_i)^2+d(y_i)^2-2\sum_{1\leq i<j\leq n}x_ix_j,~~b,d>0.
	\end{equation*}
	Then the constrained extremum problem $min_{(x_1,...,x_n)\in\Xi}~f$ has the following solution
	$$x_1=...=x_{n-1}=\frac{k}{b+1},~x_n=\frac{k}{d+1}=\frac{k(n-1)}{(b+1)d}=(b-n+2)\frac{k}{b+1}$$ provided that $d=\frac{n-1}{b-n+2}.$
\end{lem}
\begin{theorem}
	Let $F:\overline{M}\rightarrow M(c)$ be a bi-slant Riemannian map from a Riemannian manifold $\overline{M}$ to Kenmotsu space form $M(c)$ with $\xi\in \Gamma(rangeF_*)$ and $3\leq r=rankF< min\{\overline{m}, m\}.$ Then, the normalized $\mathcal{C}^{(kerF_*)^{\perp}}$-Casorati
	curvatures $\delta^{(kerF_*)^{\perp}}_{\mathcal{C}}(r-1)$ and $\hat{\delta}^{(kerF_*)^{\perp}}_{\mathcal{C}}(r-1)$ of  $(kerF_*)^{\perp}$ satisfy
	\begin{align}
		\varrho^{(kerF_*)^{\perp}}&\leq \delta^{(kerF_*)^{\perp}}_{\mathcal{C}}(r-1)+\dfrac{c-3}{4}-\dfrac{(c+1)}{2r}+\dfrac{3(c+1)}{2r(r-1)}(r_1cos^2\theta_{1}+r_2cos^2\theta_{2}) \label{5.14},\\
		\varrho^{(kerF_*)^{\perp}}&\leq \hat{\delta}^{(kerF_*)^{\perp}}_{\mathcal{C}}(r-1)+\dfrac{c-3}{4}-\dfrac{(c+1)}{2r}+\dfrac{3(c+1)}{2r(r-1)}(r_1cos^2\theta_{1}+r_2cos^2\theta_{2}) \label{5.15},
	\end{align}
	where, $\theta_{1},\theta_{2}$ are the slant angles and $2r_1, 2r_2$ are the dimensions of distributions $\D_{\theta_{1}} \text{and } \D_{\theta_{2}},$ respectively. Further, the case of equality holds in any of the above two inequalities at a point $p\in\overline{M}$ if and only if with respect to suitable orthonormal bases $\{e_1, ...,e_r\}$ on $(kerF_*)^\perp_p$ and $\{v_{r+1}, ...,v_m\}$ on $\Gamma((rangeF_*)_{F_*p})^\perp, \zeta=0.$ 
\end{theorem}

\begin{proof}
	Since $\overline{M}$ is a Riemannian manifold and $M(c)$ is the Kenmotsu space form such that  $3\leq r=rankF< min\{\overline{m}, m\}$, then from \eqref{8}, \eqref{13} and \eqref{15}-\eqref{18}, we have 
	\begin{align}\label{5.16}
		2\tau^{(kerF_*)^{\perp}}(p)&=\dfrac{c-3}{4}r(r-1)+\dfrac{c+1}{4}\Big\{2-2r+6(r_1cos^2\theta_{1}+r_2cos^2\theta_{2})\notag\\&-\sum_{\alpha=r+1}^{m}\sum_{i,j=1}^{r}\Big(\zeta_{ij}^{\alpha}\Big)^2+||trace\zeta||^2\Big\},
	\end{align} 
	using \eqref{5.10} in \eqref{5.16}, we get
	\begin{equation}\label{5.1}
		\dfrac{c-3}{4}r(r-1)+\dfrac{c+1}{4}\Big\{2-2r+6(r_1cos^2\theta_{1}+r_2cos^2\theta_{2})=r\mathcal{C}^{(kerF_*)^{\perp}}-||trace\zeta||^2+2\tau^{(kerF_*)^{\perp}}(p).
	\end{equation}
	Now, we consider following quadratic polynomial in term of the component of the second fundamental form of $F$
	\begin{align}\label{5.2}
		\mathcal{P}^{(kerF_*)^{\perp}}&=\dfrac{1}{2}r(r-1)\mathcal{C}^{(kerF_*)^{\perp}}+\dfrac{1}{2}(r-1)(r+1)\mathcal{C}^{(kerF_*)^{\perp}}\big(\mathcal{L}^{(kerF_*)^{\perp}}\big)-2\tau^{(kerF_*)^{\perp}}(p)\notag\\& + \dfrac{c-3}{4}r(r-1) +\dfrac{c+1}{4}\Big\{2-2r+6(r_1cos^2\theta_{1}+r_2cos^2\theta_{2})\Big\},
	\end{align}
	where $\mathcal{L}^{(kerF_*)^{\perp}}$ is hyperplane of $(kerF_*)^{\perp}$, without loss of generality, we assume that the hyperplane $\mathcal{L}^{(kerF_*)^{\perp}}$ is spanned by $e_1,...,e_{r-1},$ from \eqref{5.1} and \eqref{5.2}, we have
	\begin{align*}
		\mathcal{P}^{(kerF_*)^{\perp}}&=\sum_{\alpha =r+1}^{m}\Big\{2(r+1)\sum_{1\leq i<j\leq r-1}\Big(\zeta_{ij}^{\alpha}\Big)^2-2\sum_{1\leq i<j\leq r-1}\zeta_{ii}^{\alpha}\zeta_{jj}^{\alpha}+(r+1)\sum_{i=1}^{r-1}\Big(\zeta_{ir}^{\alpha}\Big)^2\notag \\& +r\sum_{i=1}^{r-1}\Big(\zeta_{ii}^{\alpha}\Big)^2+\dfrac{r-1}{2}\Big(\zeta_{rr}^{\alpha}\Big)^2\Big\},\notag\\&
		\geq \sum_{\alpha =r+1}^{m}\Big\{r\sum_{i=1}^{r-1}\Big(\zeta_{ii}^{\alpha}\Big)^2+\dfrac{r-1}{2}\Big(\zeta_{rr}^{\alpha}\Big)^2-2\sum_{1\leq i<j\leq r-1}\zeta_{ii}^{\alpha}\zeta_{jj}^{\alpha}\Big\}.
	\end{align*}
	Now, for $\alpha=r+1,...,m,$ considering the quadratic form $f_{\alpha}:\R^n\rightarrow\R,$ defined by
	$$f_{\alpha}\Big(\zeta_{11}^{\alpha},...,\zeta_{rr}^{\alpha}\Big)=r\sum_{i=1}^{r-1}\Big(\zeta_{ii}^{\alpha}\Big)^2+\dfrac{r-1}{2}\Big(\zeta_{rr}^{\alpha}\Big)^2-2\sum_{1\leq i<j\leq r}\zeta_{ii}^{\alpha}\zeta_{jj}^{\alpha},$$ then the constrained extremum problem $minf_{\alpha},$ subjected to $$\zeta^{\alpha}:\zeta_{11}^{\alpha}+...+\zeta_{rr}^{\alpha}=k^{\alpha},$$ where $k^\alpha$ is a real constant. Then from Lemma 5.1, we get
	$b=r,~ d=\frac{r-1}{2}$ and the critical points $\Big(\zeta_{11}^{\alpha},...,\zeta_{rr}^{\alpha}\Big)$ given by
	$$\zeta_{11}^{\alpha}=\zeta_{22}^{\alpha}=...=\zeta_{r-1,r-1}^{\alpha}=\dfrac{k^{\alpha}}{r+1},~\zeta_{rr}^{\alpha}=\dfrac{2k^{\alpha}}{r+1}$$ is a global minimum point, also $f_{\alpha}\Big(\zeta_{11}^{\alpha},...,\zeta_{rr}^{\alpha}\Big)=0,$ which implies that $	\mathcal{P}^{(kerF_*)^{\perp}}\geq 0.$ Hence, we have
	\begin{align}\label{5.3}
		\varrho^{(kerF_*)^{\perp}}&\leq \dfrac{1}{2}\mathcal{C}^{(kerF_*)^{\perp}}+\dfrac{r+1}{2r}\mathcal{C}^{(kerF_*)^{\perp}}\big(\mathcal{L}^{(kerF_*)^{\perp}}\big) + \dfrac{c-3}{4} -\dfrac{c+1}{2r}\notag\\&+\dfrac{3(c+1)}{2r(r-1)}(r_1cos^2\theta_{1}+r_2cos^2\theta_{2}), 	\end{align}
	for all hyperplane $\mathcal{L}^{(kerF_*)^{\perp}}$ of $(kerF_*)^\perp.$
	Similarly considering another quadratic polynomial 
	\begin{align*}
		\mathcal{Q}^{(kerF_*)^{\perp}}&=2r(r-1)\mathcal{C}^{(kerF_*)^{\perp}}-\dfrac{1}{2}(r-1)(2r-1)\mathcal{C}^{(kerF_*)^{\perp}}\big(\mathcal{L}^{(kerF_*)^{\perp}}\big)-2\tau^{(kerF_*)^{\perp}}(p) \notag\\&+ \dfrac{c-3}{4}r(r-1) +\dfrac{c+1}{4}\Big\{2-2r+6(r_1cos^2\theta_{1}+r_2cos^2\theta_{2})\Big\},
	\end{align*}	
	proceeding the similar way as above, we get $\mathcal{Q}^{(kerF_*)^{\perp}}\geq 0,$ therefore we have
	\begin{align}\label{5.4}
		\varrho^{(kerF_*)^{\perp}}&\leq 2\mathcal{C}^{(kerF_*)^{\perp}}_p-\dfrac{2r-1}{2r}\mathcal{C}^{(kerF_*)^{\perp}}\big(\mathcal{L}^{(kerF_*)^{\perp}}\big) + \dfrac{c-3}{4} -\dfrac{c+1}{2r}\\&+\dfrac{3(c+1)}{2r(r-1)}(r_1cos^2\theta_{1}+r_2cos^2\theta_{2}),
	\end{align}
	for all hyperplanes $\mathcal{L}^{(kerF_*)^{\perp}}$ of $(kerF_*)^\perp.$ Now taking the infimum in \eqref{5.3} and the supremum in \eqref{5.4} over all hyperplanes $\mathcal{L}^{(kerF_*)^{\perp}},$ we get \eqref{5.14} and \eqref{5.15}. Also analyzing the equality case for  $	\mathcal{P}^{(kerF_*)^{\perp}}\geq 0.$ and $\mathcal{Q}^{(kerF_*)^{\perp}}\geq 0,$ we get the equality conditions
	\begin{equation*}
		\begin{cases}
			\zeta^{{\alpha}}_{11}=\zeta^{{\alpha}}_{22}=...=\zeta^{{\alpha}}_{r-1r-1}=\frac{1}{2}\zeta^{{\alpha}}_{rr},\\
			\zeta^{{\alpha}}_{ij}=0,~~ i, j=1,...,r;i\neq j.
		\end{cases}
	\end{equation*}
	Since, $\zeta^{{\alpha}}_{rr}=0$ as $F_*e_r=\xi$, we conclude the result.
\end{proof}

\begin{corollary}
	Let $F:\overline{M}\rightarrow M(c)$ be a Riemannian map from a Riemannian manifold $\overline{M}$ to Kenmotsu space form $M(c)$ with $\xi\in \Gamma(rangeF_*)$ and $3\leq r=rankF< min\{\overline{m}, m\}.$ Then, for normalized $\mathcal{C}^{(kerF_*)^{\perp}}$-Casorati
	curvatures $\delta^{(kerF_*)^{\perp}}_{\mathcal{C}}(r-1)$ and $\hat{\delta}^{(kerF_*)^{\perp}}_{\mathcal{C}}(r-1)$ of  $(kerF_*)^{\perp}$ and different Riemannian maps, we have Table \ref{tab}
	\begin{table}[htbp]
		\begin{tabular}{|c|c|c|c|p{6.5cm}|}\hline
			\textbf{Riemannian Maps}&\textbf{$r_1$\& $r_2$}&$\theta_{1}$ &$\theta_{2}$&\textbf{normalized $\mathcal{C}^{(kerF_*)^{\perp}}$-Casorati
				curvatures relation}\\ \hline
			Invariant&$r_1=0$&-&$0$&\begin{tabular}{p{6.5cm}}
				\\	$\varrho^{(kerF_*)^{\perp}}\leq \delta^{(kerF_*)^{\perp}}_{\mathcal{C}}(r-1)+\frac{c-3}{4}+\frac{c+1}{r},$\\
				$\varrho^{(kerF_*)^{\perp}}\leq\hat{\delta}^{(kerF_*)^{\perp}}_{\mathcal{C}}(r-1)+\frac{c-3}{4}+\frac{c+1}{r},$
			\end{tabular}\\
			\hline
			Anti-invariant&$r_1=0$&-&$\frac{\pi}{2}$&\begin{tabular}{p{6.5cm}}
				\\	$\varrho^{(kerF_*)^{\perp}}\leq \delta^{(kerF_*)^{\perp}}_{\mathcal{C}}(r-1)+\frac{c-3}{4},$\\
				$\varrho^{(kerF_*)^{\perp}}\leq\hat{\delta}^{(kerF_*)^{\perp}}_{\mathcal{C}}(r-1)+\frac{c-3}{4},$
			\end{tabular}\\ \hline
			Semi-invariant&$r_1,r_2\neq0$&$0$&$\frac{\pi}{2}$&\begin{tabular}{p{6.5cm}}
				\\	$\varrho^{(kerF_*)^{\perp}}\leq \delta^{(kerF_*)^{\perp}}_{\mathcal{C}}(r-1)+\frac{c-3}{4}+\frac{c+1}{2r}\big\{\frac{3r_1}{r-1}-\frac{1}{2}\big\},$\\
				$\varrho^{(kerF_*)^{\perp}}\leq\hat{\delta}^{(kerF_*)^{\perp}}_{\mathcal{C}}(r-1)+\frac{c-3}{4}+\frac{c+1}{2r}\big\{\frac{3r_1}{r-1}-\frac{1}{2}\big\},$
			\end{tabular}\\ \hline
			Proper slant&$r_1=0$&-&$\in(0,\frac{\pi}{2})$&\begin{tabular}{p{6.5cm}}
				\\	$\varrho^{(kerF_*)^{\perp}}\leq \delta^{(kerF_*)^{\perp}}_{\mathcal{C}}(r-1)+\frac{c-3}{4}+\frac{c+1}{4r}(3cos^2\theta_{2}-1),$\\
				$\varrho^{(kerF_*)^{\perp}}\leq\hat{\delta}^{(kerF_*)^{\perp}}_{\mathcal{C}}(r-1)+\frac{c-3}{4}+\frac{c+1}{4r}(3cos^2\theta_{2}-1),$
			\end{tabular}\\ \hline
			Semi-slant&$r_1,r_2\neq0$&$0$&$\in(0,\frac{\pi}{2})$&\begin{tabular}{p{6.5cm}}
				\\	$\varrho^{(kerF_*)^{\perp}}\leq \delta^{(kerF_*)^{\perp}}_{\mathcal{C}}(r-1)+\frac{c-3}{4}+\frac{c+1}{2r}\big\{\frac{3(r_1+r_2cos^2\theta_{2})}{r-1}-\frac{1}{2}\big\},$\\
				$\varrho^{(kerF_*)^{\perp}}\leq\hat{\delta}^{(kerF_*)^{\perp}}_{\mathcal{C}}(r-1)+\frac{c-3}{4}+\frac{c+1}{2r}\big\{\frac{3(r_1+r_2cos^2\theta_{2})}{r-1}-\frac{1}{2}\big\},$
			\end{tabular}\\ \hline
			Hemi-slant&$r_1,r_2\neq0$&$\frac{\pi}{2}$&$\in(0,\frac{\pi}{2})$&\begin{tabular}{p{6.5cm}}\\
				$\varrho^{(kerF_*)^{\perp}}\leq \delta^{(kerF_*)^{\perp}}_{\mathcal{C}}(r-1)+\frac{c-3}{4}+\frac{c+1}{2r}\big\{\frac{3r_2cos^2\theta_{2}}{r-1}-\frac{1}{2}\big\},$\\
				$\varrho^{(kerF_*)^{\perp}}\leq\hat{\delta}^{(kerF_*)^{\perp}}_{\mathcal{C}}(r-1)+\frac{c-3}{4}+\frac{c+1}{2r}\big\{\frac{3r_2cos^2\theta_{2}}{r-1}-\frac{1}{2}\big\},$ \\
			\end{tabular}\\ \hline
		\end{tabular}
		\caption{Normalized $\mathcal{C}^{(kerF_*)^{\perp}}$-Casorati
			curvatures relation, when $\xi\in \Gamma(rangeF_*)$.}
		\label{tab}
	\end{table}\\
	where $\theta_{1},\theta_{2}$ are the slant angles and $2r_1, 2r_2$ are the dimensions of distributions $\D_{\theta_{1}} \text{and } \D_{\theta_{2}},$ respectively.
	Further, the case of equality holds in any of the above two inequalities at a point $p\in\overline{M}$ if and only if with respect to suitable orthonormal bases $\{e_1, ...,e_r\}$ on $(kerF_*)^\perp_p$ and $\{v_{r+1}, ...,v_m\}$ on $\Gamma((rangeF_*)_{F_*p})^\perp,~\zeta^{{\alpha}}=0.$
\end{corollary}

Further, if  $\xi\in \Gamma((rangeF_*)^{\perp})$, then we have the following result, which can be prove same as above theorem.
\begin{theorem}
	Let $F:\overline{M}\rightarrow M(c)$ be a bi-slant Riemannian map from a Riemannian manifold $\overline{M}$ to Kenmotsu space form $M(c)$ with $\xi\in \Gamma((rangeF_*)^\perp)$ and $3\leq r=rankF< min\{\overline{m}, m\}.$ Then, the normalized $\mathcal{C}^{(kerF_*)^{\perp}}$-Casorati
	curvatures $\delta^{(kerF_*)^{\perp}}_{\mathcal{C}}(r-1)$ and $\hat{\delta}^{(kerF_*)^{\perp}}_{\mathcal{C}}(r-1)$ of  $(kerF_*)^{\perp}$ satisfy
	\begin{align}
		\varrho^{(kerF_*)^{\perp}}&\leq \delta^{(kerF_*)^{\perp}}_{\mathcal{C}}(r-1)+\dfrac{c-3}{4}+\dfrac{3(c+1)}{2r(r-1)}(r_1cos^2\theta_{1}+r_2cos^2\theta_{2}) \label{5.6},\\
		\varrho^{(kerF_*)^{\perp}}&\leq \hat{\delta}^{(kerF_*)^{\perp}}_{\mathcal{C}}(r-1)+\dfrac{c-3}{4}+\dfrac{3(c+1)}{2r(r-1)}(r_1cos^2\theta_{1}+r_2cos^2\theta_{2}) \label{5.7},
	\end{align}
	where $\theta_{1},\theta_{2}$ are the slant angles and $2r_1, 2r_2$ are the dimensions of distributions $\D_{\theta_{1}} \text{and } \D_{\theta_{2}},$ respectively.
	Further, the case of equality holds in any of the above two inequalities at a point $p\in\overline{M}$ if and only if with respect to suitable orthonormal bases $\{e_1, ...,e_r\}$ on $(kerF_*)^\perp_p$ and $\{v_{r+1}, ...,v_m\}$ on $\Gamma((rangeF_*)_{F_*p})^\perp$, the components of $\zeta$ satisfy
	\begin{equation}\label{5.8}
		\begin{cases}
			\zeta^{{\alpha}}_{11}=\zeta^{{\alpha}}_{22}=...=\zeta^{{\alpha}}_{r-1r-1}=\frac{1}{2}\zeta^{{\alpha}}_{rr},\\
			\zeta^{{\alpha}}_{ij}=0,~~ i\neq j=1,...,r.
		\end{cases}
	\end{equation}
\end{theorem} 
\begin{corollary}
	Let $F:\overline{M}\rightarrow M(c)$ be a Riemannian map from a Riemannian manifold $\overline{M}$ to Kenmotsu space form $M(c)$ with $\xi\in \Gamma((rangeF_*)^\perp)$ and $3\leq r=rankF< min\{\overline{m}, m\}.$ Then, for normalized $\mathcal{C}^{(kerF_*)^{\perp}}$-Casorati
	curvatures $\delta^{(kerF_*)^{\perp}}_{\mathcal{C}}(r-1)$ and $\hat{\delta}^{(kerF_*)^{\perp}}_{\mathcal{C}}(r-1)$ of  $(kerF_*)^{\perp}$ and different Riemannian maps we have Table \ref{tab2}\\
	\begin{table}[htbp]
		
		\begin{tabular}{|c|c|c|c|p{6.5cm}|}\hline
			\textbf{Riemannian Maps}&\textbf{$r_1$\& $r_2$}&$\theta_{1}$ &$\theta_{2}$&\textbf{normalized $\mathcal{C}^{(kerF_*)^{\perp}}$-Casorati
				curvatures relation}\\ \hline
			Invariant&$r_1=0$&-&$0$&\begin{tabular}{p{6.5cm}}
				\\	$\varrho^{(kerF_*)^{\perp}}\leq \delta^{(kerF_*)^{\perp}}_{\mathcal{C}}(r-1)+\frac{c-3}{4}+\frac{3(c+1)}{4(r-1)},$\\
				$\varrho^{(kerF_*)^{\perp}}\leq\hat{\delta}^{(kerF_*)^{\perp}}_{\mathcal{C}}(r-1)+\frac{c-3}{4}+\frac{3(c+1)}{4(r-1)},$
			\end{tabular}\\
			\hline
			Anti-invariant&$r_1=0$&-&$\frac{\pi}{2}$&\begin{tabular}{p{6.5cm}}
				\\	$\varrho^{(kerF_*)^{\perp}}\leq \delta^{(kerF_*)^{\perp}}_{\mathcal{C}}(r-1)+\frac{c-3}{4},$\\
				$\varrho^{(kerF_*)^{\perp}}\leq\hat{\delta}^{(kerF_*)^{\perp}}_{\mathcal{C}}(r-1)+\frac{c-3}{4},$
			\end{tabular}\\ \hline
			Semi-invariant&$r_1,r_2\neq0$&$0$&$\frac{\pi}{2}$&\begin{tabular}{p{6.5cm}}
				\\	$\varrho^{(kerF_*)^{\perp}}\leq \delta^{(kerF_*)^{\perp}}_{\mathcal{C}}(r-1)+\frac{c-3}{4}+\frac{3(c+1)r_1}{2r(r-1)},$\\
				$\varrho^{(kerF_*)^{\perp}}\leq\hat{\delta}^{(kerF_*)^{\perp}}_{\mathcal{C}}(r-1)+\frac{c-3}{4}+\frac{3(c+1)r_1}{2r(r-1)},$
			\end{tabular}\\ \hline
			Proper slant&$r_1=0$&-&$\in(0,\frac{\pi}{2})$&\begin{tabular}{p{6.5cm}}
				\\	$\varrho^{(kerF_*)^{\perp}}\leq \delta^{(kerF_*)^{\perp}}_{\mathcal{C}}(r-1)+\frac{c-3}{4}+\frac{3(c+1)}{4(r-1)}cos^2\theta_{2},$\\
				$\varrho^{(kerF_*)^{\perp}}\leq\hat{\delta}^{(kerF_*)^{\perp}}_{\mathcal{C}}(r-1)+\frac{c-3}{4}+\frac{3(c+1)}{4(r-1)}cos^2\theta_{2},$
			\end{tabular}\\ \hline
			Semi-slant&$r_1,r_2\neq0$&$0$&$\in(0,\frac{\pi}{2})$&\begin{tabular}{p{6.5cm}}
				\\	$\varrho^{(kerF_*)^{\perp}}\leq \delta^{(kerF_*)^{\perp}}_{\mathcal{C}}(r-1)+\frac{c-3}{4}+\frac{3(c+1)}{2r(r-1)}(r_1+r_2cos^2\theta_{2}),$\\
				$\varrho^{(kerF_*)^{\perp}}\leq\hat{\delta}^{(kerF_*)^{\perp}}_{\mathcal{C}}(r-1)+\frac{c-3}{4}+\frac{3(c+1)}{2r(r-1)}(r_1+r_2cos^2\theta_{2}),$
			\end{tabular}\\ \hline
			Hemi-slant&$r_1,r_2\neq0$&$\frac{\pi}{2}$&$\in(0,\frac{\pi}{2})$&\begin{tabular}{p{6.5cm}}\\
				$\varrho^{(kerF_*)^{\perp}}\leq \delta^{(kerF_*)^{\perp}}_{\mathcal{C}}(r-1)+\frac{c-3}{4}+\frac{3(c+1)}{2r(r-1)}r_2cos^2\theta_{2},$\\
				$\varrho^{(kerF_*)^{\perp}}\leq\hat{\delta}^{(kerF_*)^{\perp}}_{\mathcal{C}}(r-1)+\frac{c-3}{4}+\frac{3(c+1)}{2r(r-1)}r_2cos^2\theta_{2},$ \\
			\end{tabular}\\ \hline
		\end{tabular}
		\caption{Normalized $\mathcal{C}^{(kerF_*)^{\perp}}$-Casorati
			curvatures relation, when $\xi\in \Gamma((rangeF_*)^\perp)$.}
	\label{tab2}
	\end{table}\\
	where $\theta_{1},\theta_{2}$ are the slant angles and $2r_1, 2r_2$ are the dimensions of distributions $\D_{\theta_{1}} \text{and } \D_{\theta_{2}},$ respectively. Further, the case of equality holds in any of the above two inequalities at a point $p\in\overline{M}$ if and only if with respect to suitable orthonormal bases $\{e_1, ...,e_r\}$ on $(kerF_*)^\perp_p$ and $\{v_{r+1}, ...,v_m\}$ on $\Gamma((rangeF_*)_{F_*p})^\perp$, the components of $\zeta$ satisfy\begin{equation}\label{5.8}
		\begin{cases}
			\zeta^{{\alpha}}_{11}=\zeta^{{\alpha}}_{22}=...=\zeta^{{\alpha}}_{r-1r-1}=\frac{1}{2}\zeta^{{\alpha}}_{rr},\\
			\zeta^{{\alpha}}_{ij}=0,~~ i\neq j=1,...,r.
		\end{cases}
	\end{equation}
\end{corollary}

\newpage

\bmhead{Acknowledgements}
The first author is thankful to UGC for providing financial assistance in terms of MANF scholarship vide letter with UGC-Ref. No. 4844/(CSIR-UGC NET JUNE 2019). The second author is thankful to DST Gov. of India for providing financial support in terms of DST-FST label-I grant vide sanction number SR/FST/MS-I/2021/104(C).

\section*{Declarations}

\bmhead{Author contributions} All authors contributed equally.

\bmhead{Data Availability} There is no associate data.

\bmhead{Conflict of interest} There is no conflict of interest.

\end{document}